\documentclass{amsart}

\usepackage{amsmath, amssymb, amsthm}
\usepackage{xcolor}

\newtheorem{theorem}{Theorem}
\newtheorem{lemma}[theorem]{Lemma}
\newtheorem{corollary}[theorem]{Corollary}

\newcommand{\Psu}{\mathrm{P}_{\mathcal{U}}}
\newcommand{\Fit}{\mathrm{Fit}}
\newcommand{\cp}{\mathrm{cp}}

\title[Supersoluble groups]{Supersoluble groups and the probability of generating 
	a supersoluble subgroup}
\author{Andrea Lucchini}
\address{Andrea Lucchini\\ Universit\`a di Padova\\  Dipartimento di Matematica \lq\lq Tullio Levi-Civita\rq\rq\\ Via Trieste 63, 35121 Padova, Italy\\email: lucchini@math.unipd.it}
\date{}

\begin{document}
	\maketitle
	
	\begin{abstract}
		Let $G$ be a finite group and let $\Psu(G)$ denote the 
		probability that two randomly chosen elements of $G$ generate 
		a supersoluble subgroup. We prove that if $\Psu(G) > 16/25$ 
		then $G$ is supersoluble, and that the bound $16/25$ is sharp, 
		being attained by the group $G = (C_5 \times C_5) \rtimes Q_8$, 
		where $Q_8$ acts faithfully and irreducibly on $C_5 \times C_5$.
	\end{abstract}
	
	\section{Introduction}
	
	Let $\mathcal{C}$ be a class of finite groups. We say that 
	$\mathcal{C}$ is \textit{recognizable} if, whenever all 
	$2$-generated subgroups of a finite group $G$ belong to 
	$\mathcal{C}$, then $G$ itself belongs to $\mathcal{C}$. 
	It is well known for example that the class of finite soluble groups is recognizable (see Thompson \cite{th} and Flavell \cite{pf}). On the other hand the class of finite metabelian groups is not recognizable (see \cite[Example 7.1]{bln}).
	More quantitatively, for $0 < \eta < 1$, we say that 
	$\mathcal{C}$ is \textit{$\eta$-recognizable} if, whenever 
	the probability that a $2$-generated subgroup of $G$ belongs 
	to $\mathcal{C}$ is greater than $\eta$, then $G \in \mathcal{C}$. 
	Formally, letting
	$$\mathrm{P}_{\mathcal{C}}(G) = \frac{|\{(x,y) \in G \times G 
		\mid \langle x,y\rangle \in \mathcal{C}\}|}{|G|^2},$$
	we say that $\mathcal{C}$ is $\eta$-recognizable if 
	$\mathrm{P}_{\mathcal{C}}(G) > \eta$ implies $G \in \mathcal{C}$.
	
	Clearly, every $\eta$-recognizable class is recognizable. 
	It is natural to ask whether every recognizable class 
	is $\eta$-recognizable for some $\eta$, and what the optimal 
	value of $\eta$ is.
	
	Guralnick and Wilson \cite{GW00} proved that for any class 
	$\mathcal{C}$ closed under taking subgroups, quotients and 
	extensions, $\mathcal{C}$ is $\eta$-recognizable for some 
	$\eta > 0$ (combining \cite[Proposition 5]{mqrd} with \cite[Theorem
	1.1]{mqrd}, one may deduce that $\eta$ can be taken to be $\frac{37}{90})$.
	In particular they proved that the class 
	of finite soluble groups is $\frac{11}{30}$-recognizable
	and that this bound is sharp. 
	
	Earlier, 
	Gustafson \cite{Gus73} had proved that the class of finite  abelian 
	groups is $\frac{5}{8}$-recognizable, while Guralnick and Wilson in \cite{GW00} also prove that the class of finite nilpotent groups is $\frac{1}{2}$-recognizable. The next natural candidate to investigate is the class 
	$\mathcal{U}$ of supersoluble groups, which is not closed under taking extensions and lies strictly 
	between the classes of nilpotent and soluble groups. 
	It is known that $\mathcal{U}$ is recognizable 
	(see for instance \cite[Example 1]{BH09}).
	 Our main result determines the 
	optimal threshold for $\eta$-recognizability.

	\begin{theorem}\label{thm:main}
		The class of supersoluble groups is $\frac{16}{25}$-recognizable. 
		More precisely, if $G$ is a finite group with 
		$\Psu(G) > \frac{16}{25}$, then $G$ is supersoluble. 
		Moreover, the bound $\frac{16}{25}$ is sharp, being attained 
		by the group $G = (C_5 \times C_5) \rtimes Q_8$, where $Q_8$ 
		acts faithfully and irreducibly on $C_5 \times C_5$.
	\end{theorem}
	
We note that it remains an open problem whether every 
recognizable class is $\eta$-recognizable for some 
$\eta > 0$. Natural candidates for further investigation 
are the class of groups with nilpotent derived subgroup 
and, more generally, the class of groups of Fitting length 
at most $k$, for a fixed positive integer $k$. These classes 
are known to be recognizable \cite[Examples 2 and 3]{BH09}, 
but whether they are $\eta$-recognizable for some $\eta > 0$ 
remains open.
	
	The results of Guralnick and Wilson easily reduce the proof of Theorem \ref{thm:main} to 
	 a detailed analysis of primitive soluble
	  groups with supersoluble point-stabilizers. 
	Recall that a finite group is \textit{primitive} if it has a 
	faithful primitive permutation representation; equivalently, 
	it has a core-free maximal subgroup. A primitive soluble group 
	has the form $G = V \rtimes H$, where $V = \mathrm{Soc}(G)$ is 
	elementary abelian and $H$ acts faithfully and irreducibly on $V$.
	
	The paper is organized as follows. In Section~2 we establish 
	the key lemmas on the module-theoretic structure of primitive 
	soluble groups. In Section~3 we prove the main theorem.
	
	\section{Preliminary lemmas}
	
	Throughout this section, $G = V \rtimes H$ is a primitive 
	soluble group, where $V$ is a faithful irreducible $H$-module 
	of order $p^n$, with $p$ prime and $n \geq 2$. We assume 
	moreover that $H$ is supersoluble.
	
	Since $H$ is supersoluble, $H/\Fit(H)$ is abelian. Since $V$ 
	is a faithful irreducible $H$-module, $O_p(H)=1$. Let $P$ be  a Sylow $p$-subgroup of $H.$ Since 
	$P \cap \Fit(H) = 1$, $P$ is
	isomorphic to its image in $H/\Fit(H)$, which is abelian, 
	hence $P$ is abelian. Moreover $H/\Fit(H)$ splits as
	$$H/\Fit(H) = K/\Fit(H) \times P\,\Fit(H)/\Fit(H)$$
	for a suitable subgroup $K$ of $H$, and $H = K \rtimes P$.
	
	For every pair $(h_1,h_2) \in H \times H$, let
	$$\Sigma(h_1,h_2) = \{(v_1,v_2)\in V^2 \mid 
	\langle v_1h_1, v_2h_2\rangle \text{ is supersoluble}\}$$
	and set $\pi(h_1,h_2) = |\Sigma(h_1,h_2)|/|V|^2$. Note that
	$$\Psu(G) = \frac{1}{|H|^2} \sum_{(h_1,h_2) \in H \times H} 
	\pi(h_1,h_2).$$
	
	We say that a pair $(h_1,h_2) \in H \times H$ is 
	\textit{$V$-large} if some composition factor of the 
	$\langle h_1,h_2\rangle$-module $V$ has order $\geq p^2$.
	
	\begin{lemma}\label{lem:large}
		If $(h_1,h_2)$ is $V$-large, then $\pi(h_1,h_2) \leq 1/p^2$.
	\end{lemma}
	\begin{proof}
		Let $X = \langle h_1,h_2\rangle$ and let $W_1 \leq W_2$ be 
		$X$-submodules of $V$ such that $M=W_2/W_1$ is $X$-irreducible 
		and $|W_2/W_1| \geq p^2$. Fix a pair
	$$
		(a_1+W_2,a_2+W_2)\in(V/W_2)^2.
$$
Consider the subgroup $$E_{a_1,a_2}=\langle M,\,(a_1h_1)W_1,\,(a_2h_2)W_1\rangle$$	of the group	$(V\rtimes X)/W_1$.
For every pair $(v_1,v_2)\in (a_1+W_2,a_2+W_2),$ let 
$Y=\langle v_1h_1,v_2h_2\rangle$ and $T=YW_1/W_1.$
	If	 \(Y\) is supersoluble, then \(T\) is also supersoluble.  Moreover
		\(M\cap T\) is normal in \(E_{a_1,a_2}\). Since $M$ is $X$-irreducible, either
		$M\cap T=1$ or $M\cap T=M.$ The second case cannot occur, since, being $|M|\geq p^2,$ $M$ would be a non cyclic minimal normal subgroup of the supersolvable group $T$.
	So $T$ is a 
		complement  of $M$ in $E_{a_1,a_2}$. By a theorem 
		of Gasch\"{u}tz, all complements of $M$ in $E_{a_1,a_2}$ are conjugate (see for example \cite[Proposition 16.8]{dh}), 
		so their number is $|M|=|W_2/W_1|$.
		Assume that $T=\langle v_1h_1,v_2h_2\rangle W_1/W_1$ is one of these complements and choose $v_1^\prime, v_2^\prime \in (a_1+W_2,a_2+W_2).$ If 	$\langle v_1h_1, v_2h_2\rangle W_1/W_1= 
		\langle v_1'h_1, v_2'h_2\rangle W_1/W_1$ then 
		$(v_1-v_1', v_2-v_2') \in W_2\cap T=W_1$, so each complement of $M$ in $E_{a_1,a_2}$
		is realized by at most $|W_1|^2$ pairs $(v_1,v_2) \in V^2$. 
		Therefore the number of pairs $(v_1,v_2)\in (a_1+W_2,a_2+W_2)$ for which $\langle v_1h_1,v_2h_2\rangle$ is 
		supersoluble is at most $|W_2/W_1||W_1|^2.$
We conclude
	$$
			|\Sigma(h_1,h_2)|
		\leq |V:W_2|^2\,|W_2/W_1|\,|W_1|^2
		=\frac{|V|^2}{|W_2/W_1|}\leq\frac{|V|^2}{p^2}.
\qedhere$$
	\end{proof}
	
	\begin{lemma}\label{lem:abelian}
		Let $X=\langle h_1,h_2\rangle \leq H$. If $(h_1,h_2)$ is not $V$-large, then 
		$X$ is abelian.
	\end{lemma}
	\begin{proof}
		If $(h_1,h_2)$ is not $V$-large, all the $X$-composition 
		factors of $V$ have order $p$. Let $\mathbb{F}$ be the field 
		of order $p$. Then there exists an $\mathbb{F}$-basis $(e_1,\dots,e_n)$ of $V$ 
		such that $\langle e_n,\dots,e_i\rangle$ is $X$-invariant for 
		every $1 \leq i \leq n$. With respect to this basis, $X$ 
		embeds into the group $T(n,p)$ of upper triangular $n\times n$ 
		matrices over $\mathbb{F}$. This implies that the Sylow 
		$p$-subgroup $Q$ of $X$ is normal in $X$ and has an abelian 
		complement $Y$. Since $X \leq K \rtimes P$, we have $Y \leq K$ 
		and $Q$ is abelian. Moreover 
		$[Q,Y] \leq [Q,K] \cap Q \leq K \cap Q = 1$, hence $X$ 
		is abelian.
	\end{proof}
	
	\begin{corollary}\label{cor:nonabelian}
		If $\langle h_1,h_2\rangle$ is nonabelian, then 
		$(h_1,h_2)$ is $V$-large, and in particular 
		$\pi(h_1,h_2) \leq 1/p^2$.
	\end{corollary}
	\begin{proof}
		Immediate from Lemmas~\ref{lem:large} and~\ref{lem:abelian}.
	\end{proof}

	\begin{lemma}\label{lem:q-large}
		Let $q$ be a prime power that does not divide $p-1$ and is coprime with $p$. If $q$ 
		divides the exponent of $\langle h_1,h_2\rangle$, then $(h_1, h_2)$ is $V$-large.
	\end{lemma}
	\begin{proof}
	Let $X=\langle h_1,h_2\rangle$ and let $V = W_0 \geq W_1 \geq \dots \geq W_s = 1$ be a composition 
		series of $V$ as an $X$-module. The action of $X$ on each factor 
		$W_i/W_{i+1}$ induces a homomorphism
		$$\gamma \colon X \to \prod_{0 \leq i \leq s-1} \mathrm{Aut}(W_i/W_{i+1}),$$
		where the $i$-th component of $\gamma(x)$ is the automorphism 
		$wW_{i+1} \mapsto xwx^{-1}W_{i+1}$ of $W_i/W_{i+1}$.
		The kernel of $\gamma$ acts trivially on each factor 
		$W_i/W_{i+1}$, hence embeds in the group of upper triangular 
		matrices with ones on the diagonal over $\mathbb{F}_p$, which 
		is a $p$-group. Since $q$ is a prime power coprime with $p$, 
		$q$ divides the exponent of $\mathrm{Im}(\gamma)$, hence $q$ divides 
		$|\mathrm{Aut}(W_i/W_{i+1})|$ for some $i$. Since 
		$|\mathrm{Aut}(W_i/W_{i+1})|$ divides $p-1$ when 
		$|W_i/W_{i+1}| = p$, and $q$ does not divide $p-1$, we 
		conclude that $|W_i/W_{i+1}| \geq p^2$, so $(h_1,h_2)$ 
		is $V$-large.
	\end{proof}
	
\begin{corollary}\label{cor:q-large}
	Let $q$ be a prime power that does not divide $p-1$ and is coprime with $p.$ If $q$ 
	divides $\exp(\langle h_1, h_2 \rangle)$, then 
	$\pi(h_1,h_2) \leq 1/p^2$.
\end{corollary}
\begin{proof}
	Immediate from Lemmas~\ref{lem:large} and~\ref{lem:q-large}.
	\end{proof}

	\section{Proof of the main theorem}
	
	\subsection{The primitive case}
	
	We first establish the bound for primitive soluble groups.
	
	\begin{lemma}\label{lem:primitive}
		Let $G = V \rtimes H$ be a primitive soluble group that is 
		not supersoluble, where $|V| = p^n$ with $p$ prime and 
		$n \geq 2$, and $H$ is supersoluble. Then 
		$\Psu(G) \leq 16/25$.
	\end{lemma}
	\begin{proof}
		We consider separately the cases $H$ abelian and $H$ nonabelian.
		
		\medskip
		
		\noindent\textit{Case 1: $H$ abelian.}
		Since $H$ is abelian and acts faithfully and irreducibly on $V$, 
		$H$ embeds into $\mathrm{GL}(1,p^n) \cong \mathbb{F}_{p^n}^*$, 
		which is cyclic. Hence $H$ is cyclic of some order $m$ dividing 
		$p^n - 1$.
		
		Since $V$ is a faithful irreducible $H$-module and $n \geq 2$, 
		there exists a prime $q$ dividing $m$ such that $q$ divides 
		$p^n-1$ but does not divide $p-1$. In particular $p \neq q$.
		
		Write $m = q^a \cdot r$ with $\gcd(r,q) = 1$ and $a \geq 1$. 
		Since $H$ is cyclic of order $m$, the elements $h \in H$ whose 
		order is not divisible by $q$ are exactly those belonging to 
		the unique subgroup of order $r$. Therefore the probability 
		that $q$ does not divide $|\langle h_1,h_2\rangle|$ is at most
		$$\frac{r^2}{m^2} = \frac{1}{q^{2a}} \leq \frac{1}{q^2}.$$
		If $q$ divides $|\langle h_1,h_2\rangle|$, then 
		$\pi(h_1,h_2) \leq 1/p^2$ by Corollary~\ref{cor:q-large}. 
		Therefore:
		\begin{align*}
			\Psu(G) &\leq \frac{1}{q^{2a}} + 
			\left(1-\frac{1}{q^{2a}}\right)\frac{1}{p^2}\\
			&= \frac{1}{p^2} + \frac{1}{q^{2a}}\left(1-\frac{1}{p^2}\right)\\
			&\leq \frac{1}{p^2}+\frac{1}{q^2}\left(1-\frac{1}{p^2}\right)\\
			&= \frac{1}{p^2}+\frac{1}{q^2}-\frac{1}{p^2q^2}\\
			&\leq \frac{1}{4}+\frac{1}{9}-\frac{1}{36}
			= \frac{1}{3} < \frac{16}{25},
		\end{align*}
		since $p$ and $q$ are distinct primes.

	\medskip
	\noindent\textit{Case 2: $H$ nonabelian.}		Let $\alpha$ denote the probability that two randomly chosen 
	elements of $H$ generate a nonabelian subgroup. By 
	Corollary~\ref{cor:nonabelian}, we get
	$$\mathrm{P}_{\mathcal{U}}(G) \leq \alpha \cdot \frac{1}{p^2} 
	+ (1-\alpha) \cdot 1 = 1 - \alpha\left(1-\frac{1}{p^2}\right).$$
	Let $\cp(H)$ denote the probability that two randomly chosen 
	elements of $H$ commute. By a theorem of Gustafson \cite{Gus73}, $\cp(H)\le 5/8$ for every nonabelian finite group $H$. In fact, the proof in \cite{Gus73} establishes the sharper, index-dependent inequality	$$\cp(H) \leq \frac{1}{2}\left(1 + \frac{|Z(H)|}{|H|}\right),$$
	and the bound $5/8$ is then obtained using the fact that, since $H$ is nonabelian, $H/Z(H)$ is not cyclic, so $|Z(H)|/|H|\le 1/4$. We will use this sharper, index-dependent inequality directly, without passing through the numerical bound $5/8$, also in Cases 2b and 2c below, where it is essential to keep track of the exact value of $|H:Z(H)|$.
	
		\medskip
		\noindent\textit{Case 2a: $p \geq 5$.}
 Since every pair generating an abelian 
	subgroup commutes, $1 - \alpha \leq \cp(H)$, hence 
	$\alpha \geq 3/8$. Therefore,
	$$\Psu(G) \leq 1 - \frac{3}{8}\left(1-\frac{1}{p^2}\right) 
	\leq 1 - \frac{3}{8}\cdot\frac{24}{25} = \frac{16}{25}.$$
		
		\medskip

		\noindent\textit{Case 2b: $p = 2$.}
	Since $V$ is a faithful irreducible $H$-module, $O_2(H)=1.$
	Hence $Z(H)$ has odd order and $H/Z(H)$ is not a $2$-group (otherwise a Sylow 
	$2$-subgroup of $H$ would be normal, contradicting $O_2(H)=1)$. 
	Choose $x\in H\setminus Z(H)$ with odd order and let $B=\langle Z(H), x\rangle.$

	Since $p-1=1$, every prime $q$ satisfies $q \nmid p-1$, so 
	by Corollaries~\ref{cor:nonabelian} and~\ref{cor:q-large}, 
	$\pi(h_1,h_2) \leq 1/4$ if either $\langle h_1,h_2\rangle$ 
	is nonabelian or $1 \neq \langle h_1,h_2\rangle \leq B$.
	Setting $a = |H:Z(H)|$, it follows that
	\begin{align*}
		\Psu(G) &\leq 1-\frac{3}{4}\left(\left(1-\frac{1}{2}
		-\frac{1}{2a}\right)+\frac{|B|^2-1}{|H|^2}\right)\\
		&\leq 1-\frac{3}{4}\left(\frac{1}{2}-\frac{1}{2a}+
		\frac{9|Z(H)|^2-1}{|H|^2}\right)\\
		&\leq 1-\frac{3}{4}\left(\frac{1}{2}-\frac{1}{2a}+
		\frac{8}{a^2}\right)\\
		&= 1-\frac{3}{4}\left(\frac{63}{128}+\frac{(a-32)^2}{128a^2}\right)\\&\leq  1-\frac{3}{4}\frac{63}{128}=\frac{323}{512} < \frac{16}{25}.
	\end{align*}
	
		\medskip
\noindent\textit{Case 2c: $p = 3$.}
If $|H:Z(H)| \geq 6$, then
$$\Psu(G) \leq 1-\frac{8}{9}\left(1-\frac{1}{2}
-\frac{1}{2|H:Z(H)|}\right) \leq \frac{17}{27} < \frac{16}{25}.$$
So we may assume $H/Z(H) \cong C_2 \times C_2$. Since $O_3(H)=1,$ 3 does not divide $|H|.$ Let
$$J=\{h\in H\mid h^2=1\}, \quad t=|J|.$$ 
Given $h_1, h_2\in H$, if at least one of them does not belong to $J$, then its order is divisible by a prime-power $q>2$, with $(q,3)=1.$ But then, by Corollary \ref{cor:q-large}, $\pi(h_1,h_2)\leq 1/9.$ It follows that
$$P_{\mathcal U}(G)
\leq \frac{1}{|H|^2}\left((|H|^2-t^2)\frac{1}{9}+t^2\right)=
\frac19+\frac89\left(\frac{t}{|H|}\right)^2.$$
Since $Z(H)$ 
is cyclic (otherwise, by \cite[Chapter B, Theorem 10.9, p.~177]{dh}, $H$ cannot 
have a faithful irreducible representation), in every coset $hZ(H)$ of $Z(H)$ in $H$, we have at most two elements of order 2, corresponding to the elements $z\in Z(H)$ such that $z^2=h^{-2}$. As $|H:Z(H)|=4,$ we have that $t\leq 8.$ If $|H|\geq 16,$ then $t/|H|\leq 1/2.$ Otherwise $H$ is a nonabelian group of order 8, and therefore contains at least 2 elements of order 4, and $t\leq 6.$ In either case, $t/|H|\leq 3/4.$ But then
$$P_{\mathcal U}(G)
\leq
\frac19+\frac89\frac9{16}
=\frac{11}{18}
<\frac{16}{25}.\qedhere
$$
\end{proof}

	\subsection{Reduction to the primitive case}
	
	We now prove Theorem~\ref{thm:main} by induction on $|G|$.
	
	\begin{proof}[Proof of Theorem~\ref{thm:main}]
		Let $G$ be a finite group with $\Psu(G) > 16/25$. We 
		proceed by induction on $|G|$.
		
		\medskip
		\noindent\textit{Step 1: All proper quotients of $G$ are 
			supersoluble.}
		Let $N$ be a non-trivial normal subgroup of $G$. Since every 
		$2$-generated subgroup of $G/N$ is the image of a $2$-generated 
		subgroup of $G$, we have $\Psu(G/N) \geq \Psu(G) > 16/25$. 
		By induction, $G/N$ is supersoluble.
		
		\medskip
		\noindent\textit{Step 2: $G$ has a unique minimal normal 
			subgroup.}
		Suppose $G$ has two distinct minimal normal subgroups $N_1$ 
		and $N_2$. Then $N_1 \cap N_2 = 1$, so $G$ embeds into 
		$G/N_1 \times G/N_2$. Since both $G/N_1$ and $G/N_2$ are 
		supersoluble by Step~1, and the class of supersoluble groups 
		is closed under taking subgroups and direct products, $G$ is 
		supersoluble. 
		
		\medskip
		\noindent\textit{Step 3: $G$ is soluble.}
	The probability that two random elements of $G$ generate a 
	soluble subgroup is at least $\Psu(G) > 16/25 > 11/30$. 
	By a theorem of Guralnick and Wilson \cite[Theorem A]{GW00}, 
	the class of finite soluble groups is $\frac{11}{30}$-recognizable, 
	hence $G$ is soluble.
		
		\medskip
		\noindent\textit{Step 4: $G$ is primitive soluble.}
		By Steps~1 and~3, $G$ is soluble with a unique minimal normal 
		subgroup $N$. If $N$ is contained in the Frattini subgroup $\Phi(G)$ of $G$, then $G/\Phi(G)$ is a 
		quotient of $G/N$, hence supersoluble. Since supersolubility 
		is inherited from $G/\Phi(G)$ to $G$ \cite[VI Satz 8.6]{hup},
		$G$ would be supersoluble, a contradiction. Therefore 
		$N \not\leq \Phi(G)$, so there exists a maximal subgroup $M$ 
		of $G$ with $NM = G$ and $N \cap M = 1$. Since $N$ is the 
		unique minimal normal subgroup of $G$ and $\mathrm{Core}_G(M) 
		\leq N \cap M = 1$, the subgroup $M$ is core-free, and $G$ 
		is a primitive soluble group with $\mathrm{Soc}(G) = N$.
		
		\medskip
		\noindent\textit{Step 5: Conclusion.}
		Write $G = V \rtimes H$ where $V = N$ is elementary abelian 
		of order $p^n$ and $H \cong M$ is the point stabilizer. 
		Since $G/N \cong H$ is supersoluble by Step~1, and 
		$\Psu(G) > 16/25$, Lemma~\ref{lem:primitive} implies 
		that $G$ is supersoluble, completing the induction.
		
		\medskip
		\noindent\textit{Sharpness.}
The group $G = (C_5 \times C_5) \rtimes Q_8$, where $Q_8$ 
acts faithfully and irreducibly on $C_5 \times C_5$, is not 
supersoluble, since its unique minimal normal subgroup 
$C_5 \times C_5$ has order $25 = 5^2$, which is not prime. 
However, all proper subgroups of $G$ are supersoluble: indeed, 
since $5 \equiv 1 \pmod{4}$, the field $\mathbb{F}_5$ contains 
a primitive fourth root of unity, so every proper subgroup of 
$Q_8$ acts reducibly on $C_5 \times C_5$, and hence every 
proper subgroup of $G$ is supersoluble.
Since the probability that an ordered pair of elements 
generates $G$ is $9/25$, it follows that
$$\Psu(G) = 1 - 9/25 = 16/25.$$
This shows that the bound $16/25$ is sharp.
	\end{proof}
	


\begin{thebibliography}{9}
		
		\bibitem{BH09} 
		J.~C.~Beidleman and H.~Heineken, 
		\textit{Minimal non-$\mathfrak{F}$-groups}, 
		Ric.\ Mat.\ \textbf{58} (2009), no.~1, 33--41.
		
		\bibitem{bln}T. Burness, A. Lucchini and D. Nemmi,
		\textit{On the soluble graph of a finite group},
		J. Combin. Theory Ser. A \textbf{194} (2023), Paper No. 105708, 39 pp.
		
		\bibitem{dh}  K. Doerk and T. Hawkes, \textit{Finite soluble groups}, De Gruyter Expositions in Mathematics, 4. Walter de Gruyter \& Co., Berlin, 1992.
		
		\bibitem{pf} P. Flavell, \textit{Finite groups in which every two elements generate a soluble group}, Invent. Math.
		\textbf{121} (1995) 279-285.

		
		\bibitem{Gus73} 
		W.~H.~Gustafson, 
		\textit{What is the probability that two group elements commute?}, 
		Amer.\ Math.\ Monthly \textbf{80} (1973), 1031--1034.
		
		\bibitem{GW00} 
		R.~M.~Guralnick and J.~S.~Wilson, 
		\textit{The probability of generating a finite soluble group}, 
		Proc.\ London Math.\ Soc.\ \textbf{81} (2000), no.~2, 405--427.
		
		\bibitem{hup} B. Huppert, \textit{Endliche Gruppen},  Die Grundlehren der mathematischen Wissenschaften, Band 134. Springer-Verlag, Berlin-New York, 1967. 
		
		\bibitem{mqrd} N. Menezes, M. Quick and C. Roney-Dougal,   \textit{The probability of generating a finite simple group},
		Israel J. Math. \textbf{198}(1) (2013) 371–392.
		
		\bibitem{th} J. G. Thompson, \textit{Nonsolvable finite groups all of whose local subgroups are solvable}, Bull.
		Amer. Math. Soc. \textbf{74} (1968) 383-437.
		
	\end{thebibliography}
\end{document}